\documentclass[11pt, reqno]{amsart}
\usepackage[pagewise]{lineno}
\usepackage{amsmath, amsthm, amscd, amsfonts, amssymb, graphicx, color}
\usepackage[bookmarksnumbered,
colorlinks, plainpages,
]{hyperref}
\hypersetup{colorlinks=true,linkcolor=red, anchorcolor=green, citecolor=cyan,  urlcolor=red, filecolor=magenta, pdftoolbar=true}

\usepackage{graphicx}
\usepackage{amssymb}
\usepackage{amsmath}
\usepackage{amsfonts}
\usepackage{amsthm,latexsym,color}

\usepackage{enumitem}
\setlist[enumerate,1]{label=(\arabic*)}

\usepackage{pifont}
\usepackage{bbm}
\usepackage{verbatim}
\usepackage{amsfonts}
\usepackage{hyperref}
\usepackage{graphicx}
\usepackage{tikz}
\usepackage{placeins}
\renewcommand{\theequation}{\thesection.\arabic{equation}}

\font\TenEns=msbm10 \font\SevenEns=msbm7 \font\FiveEns=msbm5
\newfam\Ensfam

\textfont\Ensfam=\TenEns \scriptfont\Ensfam=\SevenEns
\scriptscriptfont\Ensfam=\FiveEns
\def\R{{\Ens R}}

\def\Z {\mathbb Z}
\def\R {\mathbb R}
\def\C {\mathbb C}

\theoremstyle{plain}

\DeclareMathSymbol{\subsetneqq}{\mathbin}{AMSb}{36}
\newcommand{\beq}{\begin{eqnarray}}
\newcommand{\eeq}{\end{eqnarray}}
\newcommand{\bq}{\begin{equation}}
\newcommand{\eq}{\end{equation}}
\newcommand{\beqn}{\begin{eqnarray*}}
\newcommand{\eeqn}{\end{eqnarray*}}
\newcommand{\bex}{\begin{exo}}
\newcommand{\eex}{\end{exo}}
\newcommand{\ben}{\begin{enumerate}}
\newcommand{\een}{\end{enumerate}}

\newtheorem{th1}{{\bf Theorem}}[section]
\newtheorem{thm}[th1]{{\bf Theorem}}
\newtheorem{lem}[th1]{{\bf Lemma}}
\newtheorem{prop}[th1]{{\bf Proposition}}
\newtheorem{cor}[th1]{{\bf Corollary}}
\newtheorem{rem}[th1]{\bf Remark}

\newtheorem{defi}[th1]{\bf Definition}

\theoremstyle{definition}
\newtheorem*{assumption}{Assumption}
\theoremstyle{plain}

\title[Inhomogeneous nonlinear Schr\"odinger equation]{
Threshold for the existence of scattering states for
inhomogeneous nonlinear Schr\"odinger equations without gauge invariance}
\author[M. GRIRA, H. MIYAZAKI and S. TAYACHI]{}
\email{mourad.grira@fst.utm.tn}
\email{miyazaki.hayato@kagawa-u.ac.jp}
\email{slim.tayachi@fst.rnu.tn}
\begin{document}
\maketitle
\centerline{\scshape Mourad Grira}

\medskip

{\footnotesize

  \centerline{Universit\'e de Tunis El Manar, Facult\'e des Sciences de Tunis,}  \centerline{ D\'epartement de

Math\'ematiques, Laboratoire  \'Equations aux} \centerline{ D\'eriv\'ees

Partielles LR03ES04,  2092 Tunis,

Tunisia}

}
\medskip

\centerline{\scshape Hayato Miyazaki}

\medskip

{\footnotesize

  \centerline{Teacher Training Courses, Faculty of Education,} \centerline{ Kagawa University, Takamatsu,
Kagawa 760-8522, Japan,}

}

\medskip

\centerline{\scshape Slim Tayachi}

\medskip

{\footnotesize

 \centerline{Universit\'e de Tunis El Manar, Facult\'e des Sciences de Tunis,}  \centerline{ D\'epartement de

Math\'ematiques, Laboratoire  \'Equations aux} \centerline{ D\'eriv\'ees

Partielles LR03ES04,  2092 Tunis,

Tunisia}

}

\bigskip
\begin{abstract}
We consider the asymptotic behavior of solutions to inhomogeneous nonlinear Schr\"odinger equations with non-gauge-invariant nonlinearities. The spatial coefficient in the nonlinear term may have different orders of singularity at the origin and decay at infinity.
Under a coercivity condition involving the coefficient and the nonlinearity, we show that no scattering states exist for the equation below a Strauss-type threshold determined by the decay at infinity.
Our class includes nonlinearities with a dominant non-oscillatory component.
We also prove a complementary small-data scattering result above this threshold, using non-admissible Strichartz estimates in Lorentz spaces adapted to the singular coefficient.
\end{abstract}

\renewcommand{\theequation}{\thesection.\arabic{equation}}
\section{Introduction}
\setcounter{equation}{0}

In this paper we consider the Cauchy problem for the inhomogeneous nonlinear Schr\"odinger equation
\begin{align}
\begin{cases}
i\partial_t u+\dfrac12\Delta u=K(x)F(u), \quad (t,x)\in\R\times\R^N,\\
u(0,x)=u_0(x), \quad x\in\R^N,
\end{cases}
\label{11.1} \tag{INLS}
\end{align}
where $N\ge1$, $u = u(t,x)$ is a complex-valued unknown function, and $K = K(x)$ is a complex-valued spatial coefficient.  We allow $K$ to have different orders of behavior at the origin and at infinity.  More precisely, we assume
\begin{equation}
  0\le b_0<\min(2,N),\qquad b_0\le b_\infty < \infty,
  \label{b0binfty}
\end{equation}
and impose the two-scale upper bound
\begin{equation}\label{Kup}
 |K(x)|\le C_K |x|^{-b_0}\langle x\rangle^{-(b_\infty-b_0)},
 \qquad x\ne0,
\end{equation}
for some constant $C_K>0$,
where $\langle x\rangle=(1+|x|^2)^{1/2}$.  The model case is
\[
K(x)=\mu |x|^{-b},\qquad b_0=b_\infty=b,
\]
with $\mu\in\C\setminus\{0\}$.  The exponent $b_0$ governs the singularity of the coefficient at the origin, while $b_\infty$ governs its decay at spatial infinity.

Our purpose is twofold.  First, under a coercivity condition on the product $K(x)F(z)$, we prove the non-existence of scattering states
for \eqref{11.1} below a Strauss-type threshold determined by the decay exponent $b_\infty$.  Second, above the corresponding threshold, we prove a small-data scattering result by using non-admissible Strichartz estimates in Lorentz spaces.

The assumptions on the nonlinearity and on the sign of the nonlinear term are as follows.

\begin{assumption}
Let $p>1$.  We assume:
\begin{enumerate}[label=(A\arabic*), ref=A\arabic*, itemsep=2mm]
\item \label{A1} $F(0)=0$, and there exists $C>0$ such that
\[
  |F(z)-F(w)|\le C(|z|+|w|)^{p-1}|z-w|, \qquad z,w\in\C.
\]
\item \label{A2} There exist $\varphi\in\R$, $c_0>0$, and a nonnegative measurable function $\kappa$ such that
\begin{equation}\label{KFpositive}
 \operatorname{Re}\bigl(e^{i\varphi}K(x)F(z)\bigr)\ge c_0\kappa(x)|z|^p
\end{equation}
for a.e. $x\in\R^N$ and all $z\in\C$, and
\begin{equation}\label{Klower}
 \kappa(x)\ge c_1 |x|^{-b_0}\langle x\rangle^{-(b_\infty-b_0)},
 \qquad x\ne0,
\end{equation}
for some $c_1>0$.
\end{enumerate}
\end{assumption}

Assumption \eqref{A1} is used in the local theory and in the small-data scattering argument.  Assumption~\eqref{A2} is
a coercivity condition on the product $K(x)F(z)$,
used only in the non-scattering argument.
The positivity condition \eqref{KFpositive} is needed in a unified test-function argument, while the lower bound \eqref{Klower} determines the effective decay at spatial infinity and hence the threshold for nonexistence of scattering states.

We write
$U(t)=e^{\frac{it}{2}\Delta}$
for the free Schr\"odinger group associated with the linear part of \eqref{11.1}.  For $s \ge 0$, the standard Sobolev space on $\R^N$ is defined by
$H^s=H^s(\R^N)
 :=\{f\in L^2(\R^N): \|f\|_{H^s}=\|\langle \nabla\rangle^s f\|_{L^2(\R^N)}<\infty\}.
$
We also use the weighted $L^{2}$ space
\[
 \mathcal{F}H^s=\mathcal{F}H^s(\R^N)
 :=\{f\in L^2(\R^N):\langle x\rangle^s f\in L^2(\R^N)\}
\]
with norm $\|f\|_{\mathcal{F}H^s}:=\|\langle x\rangle^s f\|_{L^2}$.

Before stating the main results, we state the notion of solution.  For general background on semilinear Schr\"odinger equations and their integral formulation, see \cite{Cazenavel}.  When \eqref{Kup} and \eqref{A1} are imposed, the local $L^2$ theory is governed by the singularity exponent $b_0$; for the corresponding local theory in the model inhomogeneous case, see \cite{tayachi1}.  We set
\[
 q_0=\frac{4(p+1)}{N(p-1)+2b_0},
 \qquad
 r_{0}=\frac{N(p+1)}{N-b_0}.
\]

\begin{defi}[Solution]\label{def:solution}
Let $I$ be an interval with $0\in I$, and let $u_0\in L^2$.  We say that $u$ is a solution to \eqref{11.1} on $I$ with initial data $u_0$ if
\[
 u\in C(I;L^2)\cap L^{q_0}_{\rm loc}(I;L^{r_0,2})
\]
and
\[
 u(t)=U(t)u_0-i\int_0^t U(t-s)K F(u(s))\,ds
\]
holds in $L^2$ for every $t\in I$.
\end{defi}

\subsection{Main results}

For $0\le \beta\le2$, we define the Strauss-type exponent
\begin{equation*}
 p_{\mathrm{st}}(\beta)
 =1+
 \frac{-(N-2+2\beta)+\sqrt{(N-2+2\beta)^2+4N(4-2\beta)}}{2N}.
\end{equation*}
Equivalently, $p_{\mathrm{st}}(\beta)$ is the positive root, in $p$, of
\[
 Np^2-(N+2-2\beta)p-2=0.
\]
When $\beta=0$, this coincides with the usual Strauss exponent.
We also write
\[
 p_0(\beta)=1+\frac{4-2\beta}{N}.
\]
For $0\le\beta<\min(2,N)$, we have
\[
 1+\frac{2-\beta}{N}<p_{\mathrm{st}}(\beta)<p_0(\beta).
\]
At the formal endpoint $\beta=2$, one has $p_{\mathrm{st}}(2)=p_0(2)=1$.
In particular, for $N\ge2$, $p_{\mathrm{st}}(\beta)$ decreases to $1$ as
$\beta\uparrow2$, and the interval $1<p<p_{\mathrm{st}}(\beta)$ degenerates in this limit.

Our first result is the nonexistence of non-trivial scattering states for \eqref{11.1} below the Strauss-type threshold determined by the decay of $K$ at infinity.

\begin{thm}[Non-existence of scattering states for general potentials] \label{but}
Let $0\le b_0\le b_\infty<\min(2,N)$, and
$1<p<p_{\mathrm{st}}(b_\infty)$.
Assume \eqref{Kup}, \eqref{A1}, and \eqref{A2}.  Let $\alpha>0$ satisfy
\[
\frac{2-b_\infty}{p - 1} - \frac{N}{2} \leq \alpha.
\]
Given $u_0\in \mathcal{F}H^\alpha$, let $u$ be a solution to \eqref{11.1} on $[0,\infty)$ with initial data $u_0$.  The following statements hold:
\begin{itemize}
    \item[(i)] If $1<p\le2$ and
    \[
    \lim_{t\to\infty}\|u-U(\cdot)u_+\|_{L^\infty(t,\infty;L^2)}=0
    \]
    for some $u_+\in L^2$, then $u_+\equiv0$.

    \item[(ii)] If $2<p<p_{\mathrm{st}}(b_\infty)$ and
    \[
    \lim_{t\to\infty}
    t^{\frac{N(p-2)}{4p}}
    \|u-U(\cdot)u_+\|_{L^{\frac{4p}{N(p-2)}}(t,\infty;L^p)}=0
    \]
    for some $u_+\in L^2\cap L^{\frac{p}{p-1}}$, then $u_+\equiv0$.
\end{itemize}
The analogous statements hold backward in time.
\end{thm}

\begin{rem}
Let $2<p<p_{\mathrm{st}}(b_\infty)$ and set
\[
 q_{\ast} = \frac{4p}{N(p-2)}.
\]
The factor $t^{1/q_{\ast}}$ in Theorem \ref{but} (ii) corresponds to the linear decay order in the same spacetime norm.
The additional assumption $u_+\in L^{\frac{p}{p-1}}$ ensures that by the Hausdorff--Young inequality, the estimate
\[
  0 \neq \| \widehat{u_+} \|_{L^p} \lesssim t^{1/q_{\ast}} \| U(\cdot) u_+ \|_{L^{q_{\ast}}(t, \infty; L^{p})} \lesssim \| \widehat{u_+} \|_{L^{p}},
\]
holds for large $t$.
The weighted condition $u_+\in \mathcal{F}H^{ \sigma}$, $ \sigma> N (1/2 - 1/p)$, is a sufficient condition for $u_+\in L^2\cap L^{\frac{p}{p-1}}$.
Thus the condition in Theorem \ref{but} (ii) is the natural order needed to determine the scattering state $u_{+}$.
If one replaces $t^{1/q_{\ast}}$ by $t^\gamma$ with $\gamma<1/q_{\ast}$, then the scattering state is no longer uniquely determined by such a condition.

We also note that Theorem \ref{but} (ii) is intentionally stated in $L^p$, not in $L^{p,2}$: the proof only requires that the $L^{q_{\ast}}_tL^p_x$ norm be small on $(T,\infty)$ for sufficiently large $T$, and an $L^{p,2}$ assumption would be stronger than necessary.
\end{rem}

\begin{rem}
Theorem \ref{but} (ii) is relevant only when
$p_{\mathrm{st}}(b_\infty)>2$.  We note that $p_{\mathrm{st}}(\beta)<2$ is equivalent to
\[
3-N-2\beta<0.
\]
Thus, under $\beta<\min(2,N)$, $p_{\mathrm{st}}(\beta)<2$ for $N\ge4$, for $N=3$ with $\beta>0$, and for $N=2$ with $\beta>1/2$.  Moreover, $p_{\mathrm{st}}(\beta)=2$ when $(N,\beta)=(3,0)$, $(N,\beta)=(2,1/2)$, while $p_{\mathrm{st}}(\beta)>2$ for $N=1$ and for $N=2$ with $0\le\beta<1/2$.
\end{rem}

\begin{rem}[Role of the weighted assumption]\label{rem:role-weighted-assumption}
Set
\[
 s_\infty:=\frac{2-b_\infty}{p-1}-\frac N2.
\]
The condition $s_\infty\le\alpha$ in Theorem \ref{but} is not a
regularity assumption for the local theory.
It is used only to control the term involving the initial data in the test-function argument; see Lemma
\ref{lem:initial-data-term}.
The exponent $s_\infty$ is the scaling-critical exponent in the weighted
scale determined by the decay of the coefficient at spatial infinity.
In the model case $K(x)=|x|^{-b}$, it
reduces to
\[
 s_c=\frac{2-b}{p-1}-\frac N2.
\]
This is the critical exponent for the homogeneous weighted scale associated with
$\mathcal{F}H^{s_c}$.
Thus Theorem \ref{but} covers the
scale-critical and supercritical range in the weighted scale determined by the behavior of $K$ at infinity.
In particular, compactly supported smooth initial data are
included in the class allowed by Theorem \ref{but}.
\end{rem}

\begin{rem}\label{rem:weak-formulation-only}
The Lipschitz-type condition in \eqref{A1} is not essential for proving
the nonexistence of scattering states.
The proof of Theorem \ref{but} only uses the weak formulation of the
equation together with the coercivity condition in \eqref{A2}.
As noted in Remark~\ref{rem:role-weighted-assumption}, the weighted assumption $u_0\in\mathcal{F}H^\alpha$ is used only to control the term involving the initial data in Lemma \ref{lem:initial-data-term}.
In particular, the proof does not require persistence of regularity in $\mathcal{F}H^\alpha$.
We keep \eqref{A1} in the main assumptions because it is needed for the local theory and for the small-data scattering result.
\end{rem}

A distinctive feature of Theorem \ref{but} is that the threshold is governed by the decay of the coefficient at spatial infinity, not by its singular behavior near the origin.  This is because the test-function estimate uses \eqref{Klower} through the behavior of $K$ as $|x|\to\infty$.
In the non-gauge-invariant case, the $L^2$-norm is not conserved in general.  Thus the global existence theory differs from the gauge-invariant case, while Theorem \ref{but} concerns the non-scattering mechanism.

The second result establishes the complementary small-data scattering above the same Strauss-type threshold.  The proof follows the argument of Aloui and the third author~\cite{tayachi2}. The smallness condition is imposed on the free evolution, following the approach of Cazenave and Weissler \cite{Cazenavewissler} as well as Kato's framework for nonlinear Schr\"odinger equations \cite{Kato94}. Smallness in $\mathcal{F}H^1$ is one sufficient condition for this assumption.
For this result,
under \eqref{b0binfty},
we impose the following compatibility condition on the parameters $b_0$,
$b_\infty$, and $p$:
\begin{equation}\label{scatter-compatibility}
 p_{\mathrm{st}}(b_{\ast})<p\le p_0(b_0),
 \qquad b_{\ast}:=\min\{2,N,b_\infty\}.
\end{equation}
This condition means that the two-scale bound \eqref{Kup} allows us to
choose at least one single-power exponent $b$ in the range required for
the small-data scattering argument.  More precisely, given
\eqref{scatter-compatibility}, one can choose $b$ so that
\begin{equation}\label{bchoice}
 b_0\le b\le b_\infty,
 \qquad
 0\le b< \min(2,N),
 \qquad
 p_{\mathrm{st}}(b)<p\le p_0(b).
\end{equation}
Then the bound \eqref{Kup} implies
\[
 |K(x)|\lesssim |x|^{-b},
 \qquad x\ne0.
\]
The equivalence between \eqref{scatter-compatibility} and the existence
of $b$ satisfying \eqref{bchoice} is explained in
Remark~\ref{rem:choice-of-b}.

\begin{thm}[Scattering for small free evolution]\label{thm:small-s}
Assume \eqref{Kup}, \eqref{A1}, and \eqref{scatter-compatibility}.
Let $b$ satisfy \eqref{bchoice}.
Set
\begin{equation}\label{main-scatter-exponents}
 q_b=\frac{4(p+1)}{N(p-1)+2b},
 \qquad
 r_b=\frac{N(p+1)}{N-b},
 \qquad
 a_b=\frac{2(p-1)(p+1)}{4-2b-(N-2)(p-1)}.
\end{equation}
Then there exists $\varepsilon_0>0$ with the following property.  If $u_0\in L^2$ satisfies
\[
 \|U(t)u_0\|_{L^{a_b}(0,\infty;L^{r_{b},\infty})}\le \varepsilon_0,
\]
then the corresponding solution to \eqref{11.1} is global forward in time and scatters in $L^2$ as $t\to\infty$.  More precisely, there exists $u_+\in L^2$ such that
\[
 \lim_{t\to\infty}\|u(t)-U(t)u_+\|_{L^2}=0.
\]
The analogous statement holds backward in time.  In particular, if
\[
 \|U(t)u_0\|_{L^{a_b}(\R;L^{r_b,\infty})}\le \varepsilon_0,
\]
then the solution is global and scatters in $L^2$ as $t\to\pm\infty$.
\end{thm}

\begin{cor}[Small $\mathcal{F}H^1$ data]\label{cor:FcalH1-small-scattering}
Assume the hypotheses of Theorem \ref{thm:small-s}.  There exists $\varepsilon_1>0$ such that, if
\[
 u_0\in \mathcal{F}H^1,
 \qquad
 \|u_0\|_{\mathcal{F}H^1}\le \varepsilon_1,
\]
then the solution to \eqref{11.1} is global in both time directions and scatters in $L^2$.
\end{cor}

\begin{rem}[Comparison with the gauge-invariant case]\label{rem:comparison-aloui-tayachi-prop47}
We compare Theorem~\ref{thm:small-s} with the small-data scattering result
in \cite[Proposition~4.7]{tayachi2}.  In the gauge-invariant case
\[
 i\partial_t u+\Delta u=K(x)|u|^{p-1}u,
\]
they considered potentials with different behaviors at the origin and at infinity, satisfying the two-scale bound \eqref{Kup} with $0<b_0<\min (2,N)$ and $b_{\infty}>b_0$, together with the corresponding bound on $\nabla K$ when $s=1$.
They proved small-data scattering in $H^s$
in the range
\[
 p_{\mathrm{st}}(\min (2,b_{\infty}))< p < 1+ \frac{4-2b_0}{N-2s}
\]
for $s=0$ when $N\ge2$, and for $s=1$ when $N\ge4$.
Equivalently, the lower threshold is determined
by the decay at infinity, while the upper Sobolev-subcritical bound is determined
by the singular behavior at the origin.

Theorem~\ref{thm:small-s} is in the same spirit concerning the two spatial
scales, but it addresses a different structural situation.  We allow homogeneous
nonlinearities which are not necessarily gauge invariant, and the result is
formulated at the $L^2$ level.
Thus the upper range is restricted by the $L^2$-critical exponent associated with the behavior of the coefficient near the origin, while the Strauss-type lower threshold is governed by its decay at infinity.
Since no gauge invariance is assumed, we do not use the pseudo-conformal structure that leads to scattering in $\Sigma := H^{1} \cap \mathcal{F}H^{1}$ in the gauge-invariant setting.
A counterpart in $H^1$ would require additional assumptions, for instance on $\nabla K$ and on the differentiability of $F$.
Scattering in $\Sigma$ is a separate issue, since it relies on pseudo-conformal structure available in the gauge-invariant setting, see for example \cite{grira1}.
We therefore restrict Theorem~\ref{thm:small-s} to the $L^2$ level.
\end{rem}

\begin{rem}
The two main results use different sides of the assumptions.  Theorem \ref{thm:small-s} uses the upper bound \eqref{Kup} and the pointwise estimate in \eqref{A1}.
Theorem \ref{but} uses the coercive lower condition \eqref{A2}, in particular \eqref{KFpositive} and \eqref{Klower}.  Thus, for the model coefficient $K(x)=\mu|x|^{-b}$, Theorems \ref{but} and \ref{thm:small-s} give the following picture: below $p_{\mathrm{st}}(b)$, nonzero scattering states are excluded under \eqref{A2}, while above $p_{\mathrm{st}}(b)$, small linear profiles give $L^2$-scattering under \eqref{A1}.  The threshold case $p=p_{\mathrm{st}}(b)$ is not addressed here.
\end{rem}

\begin{rem}
For the model equation \eqref{11.1} with $K(x)=|x|^{-b}$ and $F(z)=|z|^p$, the exponent $p_0(b)$ is the $L^2$-critical exponent.  Indeed, the scaling
\[
 u_{\tau}(t,x):= \tau^{\frac{2-b}{p-1}}u(\tau^{2}t,\tau x),
 \qquad \tau>0,
\]
leaves the equation \eqref{11.1} invariant and gives
\[
 \| \tau^{\frac{2-b}{p-1}}u_{0}(\tau \cdot) \|_{L^{2}(\R^{N})}
 =
 \tau^{\frac{2-b}{p-1}-\frac{N}{2}}
 \|u_{0} \|_{L^{2}(\R^{N})}.
\]



\end{rem}

\begin{rem}[Choice of the auxiliary exponent]\label{rem:choice-of-b}
Recall that
\[
 b_{\ast}:=\min\{2,N,b_\infty\}.
\]
For fixed $p>1$, the inequality $p_{\mathrm{st}}(b)<p$ is equivalent to
\[
 b>B_s(p):=\frac{(N+2)p+2-Np^2}{2p},
\]
while $p\le p_0(b)$ is equivalent to
\[
 b\le B_0(p):=\frac{4-N(p-1)}{2}.
\]
Hence, the compatibility condition \eqref{scatter-compatibility} is equivalent to
\begin{align}
 B_s(p)<b_{\ast},
 \qquad
 b_0\le B_0(p).
 \label{com:c1}
\end{align}
Moreover,
\[
 B_0(p)-B_s(p)=\frac{p-1}{p}>0.
\]
Thus the interval $(B_s(p),B_0(p)]$ is nonempty.
Since the condition \eqref{bchoice} is rewritten as
\[
 b_0\le b\le b_\infty,
 \qquad
 b<\min(2,N),
 \qquad
 B_s(p)<b\le B_0(p),
\]
the condition \eqref{com:c1} is equivalent to the existence
of $b$ satisfying \eqref{bchoice}.
Indeed, if $b$ satisfies \eqref{bchoice}, two inequalities \eqref{com:c1} follow
immediately.  Conversely, assume these \eqref{com:c1}.  If
$B_s(p)<b_0$, then $b=b_0$ satisfies \eqref{bchoice}.  If
$b_0\le B_s(p)$, then we can choose
\[
 B_s(p)<b<\min\{b_{\ast},B_0(p)\},
\]
because $B_s(p)<b_{\ast}$ and $B_s(p)<B_0(p)$.
This choice also satisfies $b_0\le b \le b_\infty$,
$b<\min(2,N)$, and $b\le B_0(p)$. Hence $b$ exists satisfying \eqref{bchoice}.

The condition \eqref{bchoice}
illustrates the different roles of $b_0$ and $b_\infty$: the local
theory is governed by the singularity of $K$ at the origin, while the
Strauss-type lower threshold is controlled by the decay of $K$ at
infinity, up to the natural restriction $b<\min(2,N)$.
\end{rem}

\subsection{Examples and comparison with previous works}

We next explain the class of nonlinearities covered by our assumptions.  Equation \eqref{11.1} with $b>0$ and $K(x)=|x|^{-b}$ is related to the inhomogeneous Gross--Pitaevskii equation, which describes beam propagation in an inhomogeneous medium in nonlinear optics; see \cite{Saanouni}.  For additional details on the related physical background, we refer the reader to \cite{pitaveski}.

The simplest model satisfying \eqref{A1} and \eqref{A2} is
\[
K(x)=\mu |x|^{-b},\qquad F(z)=\eta |z|^p,
\]
with $\mu\eta\ne0$ and with $\varphi$ chosen so that
\[
\operatorname{Re}(e^{i\varphi}\mu\eta)>0.
\]
The gauge-invariant case $F(z)=|z|^{p-1}z$ will be treated mainly for comparison with known scattering results.

The assumptions also cover homogeneous nonlinearities whose angular profiles have absolutely convergent Fourier expansions (cf. \cite{miyazaki3, MasakiMiyazakiUriya2019,  miyazaki2}).
Suppose that
\begin{equation}  \label{degre}
 F(\tau z)= \tau^{p} F(z), \qquad \tau>0,\ z\in\C.
\end{equation}
Then, using the relation $F(z)= |z|^pF(z/|z|)$, we identify $F$ with the $2\pi$-periodic function
\[
 g_F(\theta) := F(e^{i\theta}).
\]
We call the function $g_{F}$ the angular profile of $F$.
When $g_F$ has absolutely summable Fourier coefficients, we write
\begin{equation} \label{Fnu}
 F(z)=|z|^p\sum_{n\in\Z}a_n e^{in\arg z}
\end{equation}
for $z\ne0$, where
\[
 a_{n}= \frac{1}{2 \pi}\int_{0}^{2 \pi}F(e^{i\theta})e^{-in\theta}\,d\theta.
\]
The mode $n=1$ gives the gauge-invariant term $|z|^{p-1}z$, while the mode $n=0$ gives the non-oscillatory term $|z|^p$.
For the distinction between the gauge-invariant and non-gauge-invariant cases, see \cite{ikeda2}.
Under the condition \eqref{degre}, the Lipschitz-type condition in \eqref{A1} is equivalent to the Lipschitz continuity of $g_F$; see \cite{miyazaki3}.  Hence \eqref{A1} holds, for instance, if
\begin{equation}
  \sum_{n\in\Z}(1+|n|)|a_n|<\infty.
  \label{fourier:con}
\end{equation}
Moreover, when
\begin{equation} \label{D}
 |a_{0}| > \sum_{n \neq 0} |a_{n}|,
\end{equation}
the non-oscillating component is dominant.  In particular, if
\[
K(x)=e^{i\omega}k(x),\qquad
k(x)\ge c |x|^{-b_0}\langle x\rangle^{-(b_\infty-b_0)}
\]
for some $c>0$ and $\omega \in \R$, then \eqref{A2} holds after a suitable choice of phase.  Indeed, choosing
$\varphi=-\omega-\arg a_0$, we obtain
\begin{align*}
 \operatorname{Re}\left(e^{i\varphi}K(x)F(z)\right)
 &\ge k(x)\left(|a_0|-\sum_{n\ne0}|a_n|\right)|z|^p.
\end{align*}
Thus \eqref{D} is a concrete sufficient condition for the coercivity required in
Theorem \ref{but}.
Theorems \ref{but} and \ref{thm:small-s}, however,
are formulated in terms of \eqref{A1} and \eqref{A2} and are not restricted to
nonlinearities of the form \eqref{Fnu}.

The class also contains angular profiles that are not trigonometric
polynomials.  Let $0<\varepsilon\ll1$ and define
\[
 F_\varepsilon(z)
 =
 |z|^p\exp\left(\varepsilon\frac{{\rm Re}\,z}{|z|}\right),
 \qquad z\ne0,
\]
with $F_\varepsilon(0)=0$.  Its angular profile is
\[
 g_\varepsilon(\theta)=e^{\varepsilon\cos\theta}.
\]
This profile is not a trigonometric polynomial for $\varepsilon\ne0$.  We use
the Fourier expansion
\[
 e^{\varepsilon\cos\theta}
 =
 \sum_{n\in\Z} I_{|n|}(\varepsilon)e^{in\theta},
\]
where $I_n$ denotes the modified Bessel function of the first kind, defined by
\[
 I_n(s)
 :=
 \frac1{2\pi}\int_{-\pi}^{\pi}
 e^{s\cos\theta}e^{-in\theta}\,d\theta,
 \qquad n\in\Z,\quad s>0;
\]
see, for example, \cite[Chapter~10]{NISTHandbook}.
Thus the Fourier coefficients are $a_n := I_{|n|}(\varepsilon)$.
They decay rapidly in $n$, and in particular $\{a_{n}\}$ satisfies
\eqref{fourier:con}.
Furthermore,
\[
 a_0=I_0(\varepsilon),
 \qquad
 \sum_{n\ne0}|a_n|
 =
 e^\varepsilon-I_0(\varepsilon).
\]
Since
\[
 I_0(\varepsilon)=1+O(\varepsilon^2),
 \qquad
 e^\varepsilon-I_0(\varepsilon)=\varepsilon+O(\varepsilon^2)
\]
as $\varepsilon\to0$, condition \eqref{D} holds for all sufficiently small
$\varepsilon>0$.  Hence \eqref{A1} holds.

The coefficient $K$ may also be chosen in a non-model form.  For example, let
\[
 k(x)
 =
 \left(
 |x|^{b_0}+(2+\sin x_1)|x|^{b_\infty}
 \right)^{-1},
 \qquad x\ne0,
\]
and set $K(x)=e^{i\omega}k(x)$.
Since $1\le 2+\sin x_1\le3$,
we have
\[
 c_0 |x|^{-b_0}\le k(x)\le C_0 |x|^{-b_0}
 \quad (0<|x|\le1),
 \qquad
 c_\infty |x|^{-b_\infty}\le k(x)\le C_\infty |x|^{-b_\infty}
 \quad (|x|\ge1),
\]
for some positive constants $c_0,C_0,c_\infty,C_\infty$.
Thus this example satisfies the required two-scale bounds and is non-radial.
Moreover, combining this choice of $K$ with the above
$F_\varepsilon$,
the coercivity condition \eqref{A2} holds with the phase $\varphi=-\omega$.

We now compare our results with previous works.  We first discuss the
scattering part.  For the gauge-invariant nonlinearity
$F(u)=|u|^{p-1}u$, Farah and Guzm\'an \cite{farah} proved scattering in
$H^1$ for the focusing radial cubic case in three dimensions with
$0<b<1/2$.  This result was later extended to the non-radial setting by
Miao, Murphy, and Zheng \cite{miao}.  In the defocusing case, Dinh
\cite{dinh2,dinh} established scattering in $H^1$ and in the weighted
space $\Sigma := H^1\cap\mathcal{F}H^1$ in several energy-subcritical ranges.

Aloui and the third author \cite{tayachi2}, and Aloui, the first author,
and the third author \cite{grira1}, proved scattering in $\Sigma$ for the
defocusing gauge-invariant case in the Strauss-type range
\[
 p_{\mathrm{st}}(b)\le p<1+\frac{4-2b}{N}.
\]
In addition, Inui et al.\ \cite{aoki} obtained $L^2$-scattering for the
inhomogeneous equation, with and without an inverse-square potential,
provided
\[
 p>1+\frac{2-2b}{N}.
\]
These results are the main comparison point for the scattering part of
the present paper.  Our scattering result is formulated instead for
non-gauge-invariant nonlinearities under a smallness condition on the
free evolution, and the proof relies on non-admissible Strichartz
estimates in Lorentz spaces.

It is also useful to distinguish the present coercive case from the
oscillatory cases in which the non-oscillatory mode is absent.
To make this comparison, we recall the standard
nonlinear Schr\"odinger equation
\begin{equation}\label{eq:standard-general-NLS}
 i\partial_t u+\frac12\Delta u=F(u).
\end{equation}
For the standard equation \eqref{eq:standard-general-NLS}, Kawamoto, Masaki and the second author \cite{KMM25} proved small-data scattering in the
weighted space $\Sigma$ for certain
mass-subcritical exponents, including regimes at or below the Strauss
exponent, for nonlinearities of the form \eqref{Fnu} under the structural
condition that all nonpositive Fourier modes vanish.  In particular, the
zeroth mode is absent in their setting.  By contrast, the coercivity
condition \eqref{A2} is satisfied, for instance, when the zeroth mode is
dominant in the sense of \eqref{D}.  Theorem \ref{but} shows that a
dominant non-oscillatory mode precludes nontrivial scattering below the
Strauss-type threshold, provided a global solution with the stated
asymptotic behavior exists.

We next turn to the non-scattering part.
For the standard equation \eqref{eq:standard-general-NLS} with gauge-invariant nonlinearity $F(u)=\eta |u|^{p-1}u$,
$\eta\in\C\setminus\{0\}$, Cazenave
\cite[Theorem 7.5.2]{Cazenavel} proved the nonexistence of nontrivial
scattering states in $\Sigma=H^1\cap\mathcal{F}H^1$ when
\[
 1<p\le \min(2,1+2/N),
\]
using the pseudo-conformal transform.  In the defocusing case, this range
had already been treated by Barab \cite{Barab84} through decay estimates
derived from the pseudo-conformal energy identity.

For the inhomogeneous gauge-invariant model
\[
 K(x)=\lambda |x|^{-b},
 \qquad
 F(z)=|z|^{p-1}z,
 \qquad
 \lambda\in\R\setminus\{0\},
\]
a non-scattering result for solutions in $\Sigma$ is known in the range
\[
 1<p\le 1+\frac{2-2b}{N}
\]
under suitable restrictions on $N$ and $b$; see \cite{grira1}.  These
gauge-invariant results rely on the pseudo-conformal structure.  By
contrast, Theorem \ref{but} uses the coercive non-gauge-invariant
component in \eqref{A2}, and the relevant threshold is the Strauss-type
exponent $p_{\mathrm{st}}(b_\infty)$ determined by the decay of the coefficient at
spatial infinity.

For the non-gauge-invariant standard equation \eqref{eq:standard-general-NLS},
the second author and Sobajima \cite{miyazakinew} identified the Strauss
exponent as a threshold for the existence of scattering states in the case
$K\equiv1$, under assumptions corresponding to \eqref{A1} and \eqref{A2}.
Their method is based on a test-function argument and applies to
nonlinearities containing a non-oscillating component below the Strauss
exponent.  The proof of Theorem \ref{but} follows this line of argument,
but incorporates the two-scale inhomogeneous coefficient $K$ and shows
that the threshold is governed by the behavior of $K$ at spatial infinity.

Related blow-up and lifespan problems have also been studied for
non-gauge-invariant nonlinear Schr\"odinger equations.  For the standard
equation \eqref{eq:standard-general-NLS} and related non-gauge-invariant
models, nonexistence, blow-up, and lifespan problems have been studied in
\cite{ikeda1,fijiwara2,fijiwara1,RenLi}.  For the inhomogeneous pure-power
case $F(u)=|u|^p$, Saanouni \cite{Saanouni} derived an upper bound for
the maximal existence time $T_{\max}$ for large initial data
$f\in L^1_{\mathrm{loc}}$ with $\operatorname{Re} f=0$, in the range
\[
 1<p<1+\frac{4-2b}{N}.
\]
Finally, for general homogeneous nonlinearities of the form \eqref{Fnu},
the first and third authors \cite{grira2} obtained blow-up results and
upper-bound estimates for $T_{\max}$ for both small and large data under
the conditions
\[
 1<p<1+\frac{4-2b}{N},
 \qquad
 1-\sum_{n\ne0}|a_n|>0.
\]
These lifespan results are closely related in spirit to the present
non-scattering problem.
However, even for the model coefficient
$K(x)=|x|^{-b}$, they do not determine the lower side of the Strauss-type scattering threshold for non-gauge-invariant nonlinearities.
The lower side of the threshold was first identified for the standard equation \eqref{eq:standard-general-NLS} by the second author and Sobajima \cite{miyazakinew}.
Theorem \ref{but} extends this viewpoint to two-scale inhomogeneous coefficients and shows that the threshold is determined by the decay of $K$ at spatial infinity.

The paper is organized as follows.
In Section \ref{sec:pre}, we collect the notation and preliminary estimates used throughout the paper.
In Section \ref{sec:scattering}, we prove Theorem \ref{thm:small-s}
using a non-admissible Strichartz estimate in Lorentz spaces.
In Section \ref{sec:pfm}, we prove Theorem \ref{but} by the unified test-function method.
Finally, in Appendix \ref{app:Lorentz-nonadmissible-Strichartz}, we prove the non-admissible Strichartz estimate in Lorentz spaces used in the scattering argument.

\section{Preliminaries}\label{sec:pre}

Throughout the paper, $p'$ denotes the H\"older conjugate exponent of $p \ge 1$, that is, $1/p+1/p'=1$.
We use $C$ to denote positive constants which may change from line to line.  For a set $A\subset\R$, $\mathbbm{1}_A$ denotes its characteristic function.
$\mathcal{S}'(\R^N)$ is the space of tempered distributions on $\R^N$.
Let $C_0^\infty(\R^N)$ be the space of smooth compactly supported functions on $\R^N$.
$\mathcal{F}[u] = \widehat{u}$ is the usual Fourier transform of a function $u$ on $\R^N$, and
$\mathcal{F}^{-1}[u] = \check{u}$ is its inverse.

\subsection{Lorentz and Sobolev--Lorentz spaces}

We recall the definition and basic properties of Lorentz spaces.  For background, see \cite{grafakos,hunt,lemari}.  For a measurable function $f$ on $\R^N$, let
\[
 d_f(\lambda)=\left|\{x\in\R^N: |f(x)|>\lambda\}\right|,
 \qquad \lambda>0,
\]
and define its decreasing rearrangement by
\[
 f^*(s)=\inf\{\lambda>0: d_f(\lambda)\le s\},
 \qquad s>0.
\]
For $0<p<\infty$ and $0<q\le\infty$, the Lorentz space $L^{p,q}(\R^N)$ is the space of measurable functions $f$ such that $\|f\|_{L^{p,q}}<\infty$, where
\[
\|f\|_{L^{p,q}} =
\begin{cases}
\displaystyle
 \left(\frac{q}{p}\int_0^\infty
 \left(s^{1/p}f^*(s)\right)^q \frac{ds}{s}
 \right)^{1/q},
 & 0<q<\infty,\\[2mm]
\displaystyle
 \sup_{s>0}s^{1/p}f^*(s),
 & q=\infty.
\end{cases}
\]
Throughout the paper, when $1<p<\infty$ and $1\le q\le\infty$, we use Banach norms on $L^{p,q}$ equivalent to the usual Lorentz quasi-norms, for instance those defined in terms of the maximal rearrangement $f^{**}$.

We shall use the following generalized H\"older inequality in Lorentz spaces; see \cite[Theorems 3.4 and 3.5, p. 141]{ONeil} and \cite[Proposition 2.3, p. 19]{lemari}.

\begin{prop}[Generalized H\"older inequality in Lorentz spaces]
Let $p_1,p_2 \in (1, \infty)$ and $q_1,q_2 \in [1, \infty]$.  Suppose that
\[
 \frac1p=\frac1{p_1}+\frac1{p_2},
 \qquad
 \frac1q\le \frac1{q_1}+\frac1{q_2},
\]
where $1\le p<\infty$ and $1\le q\le\infty$.  Then
\[
 \|fg\|_{L^{p,q}}
 \le C\|f\|_{L^{p_1,q_1}}\|g\|_{L^{p_2,q_2}}.
\]
The same estimate is valid when one factor belongs to $L^\infty$.
\end{prop}

For $s\ge0$, $1<p<\infty$, and $1\le q\le\infty$, we define the Sobolev--Lorentz spaces by
\[
W^{s,p}_q(\R^N)
:=
\left\{
f\in\mathcal S'(\R^N):
(I-\Delta)^{s/2}f\in L^{p,q}(\R^N)
\right\}, \quad (I- \Delta)^{s/2} := \mathcal{F}^{-1} \langle \xi \rangle^{s} \mathcal{F},
\]
and
\[
\dot W^{s,p}_q(\R^N)
:=
\left\{
f\in\mathcal S'(\R^N):
(-\Delta)^{s/2}f\in L^{p,q}(\R^N)
\right\}, \quad (-\Delta)^{s/2} := \mathcal{F}^{-1} |\xi|^{s} \mathcal{F};
\]
see \cite[p.~571]{hjaij}. We equip these spaces, respectively, with the norms
\[
\|f\|_{W^{s,p}_q}
:=
\|(I-\Delta)^{s/2}f\|_{L^{p,q}},
\qquad
\|f\|_{\dot W^{s,p}_q}
:=
\|(-\Delta)^{s/2}f\|_{L^{p,q}}.
\]
We shall use the homogeneous Sobolev--Lorentz embedding
\[
 \dot W^{s,p}_q(\R^N)\hookrightarrow L^{\widetilde p,q}(\R^N),
 \qquad
 \frac1{\widetilde p}=\frac1p-\frac{s}{N},
\]
where $1<p<\infty$, $1\le q\le\infty$, and $0<s<N/p$; see \cite[Theorem 2.4, p. 20]{lemari}.  We also use the Sobolev--Lorentz embedding
\[
 H^s(\R^N)\hookrightarrow L^{r,2}(\R^N)
\]
whenever $s\ge0$, $2\le r<\infty$, and
\[
 \frac12-\frac{s}{N}\le \frac1r\le \frac12.
\]

\subsection{Strichartz estimates}

We say that a pair $(q,r)$ is admissible if
\[
  q, r \in [2, \infty],\quad
\frac{2}{q} = N\left(\frac12-\frac{1}{r} \right), \quad (N,q, r) \neq (2,2,\infty).
\]

\begin{prop}[Admissible Strichartz estimates in Lorentz spaces]\label{prop:admissible-strichartz}
Let $(q,r)$ and $(\widetilde{q},\widetilde{r})$ be admissible pairs with $r, \widetilde{r}<\infty$.  Then
\[
 \|U(t)\phi\|_{L^q(\R;L^{r,2})}
 \le C\|\phi\|_{L^2}
\]
for every $\phi\in L^2$.  Moreover, for every interval $I\subset\R$, every $t_0\in\overline I$, and every $f\in L^{\widetilde{q}'}(I;L^{\widetilde{r}',2})$,
\[
 \left\|\int_{t_0}^t U(t-s)f(s)\,ds\right\|_{L^q(I;L^{r,2})}
 \le C\|f\|_{L^{\widetilde{q}'}(I;L^{\widetilde{r}',2})}.
\]
\end{prop}

We next state the non-admissible Strichartz estimate used in the small-data scattering argument.  Estimates of this type go back to the work of Cazenave and Weissler \cite{Cazenavewissler} and to Kato's framework for nonlinear Schr\"odinger equations \cite{Kato94}.
Related inhomogeneous estimates can be found in \cite{Foschi,Vilela,Taggart}.  The Lorentz-space form below is adapted to the singular spatial coefficient in the present paper.  Related Lorentz-type estimates also appear in the mass-subcritical scattering theory; see \cite{NakanishiOzawa,Masaki2015}.
Set
\[
2^{\ast}=
\begin{cases}
\infty, & N=1,2,\\[4pt]
\dfrac{2N}{N-2}, & N\ge3.
\end{cases}
\]

\begin{defi}[Lorentz acceptable pairs]\label{def:Lorentz-acceptable}
A pair $(\sigma,\rho)$ is called Lorentz acceptable if
\[
2<\rho<2^{\ast},
\qquad
0<\frac1\sigma<N\left(\frac12-\frac1\rho\right).
\]
For such a pair, we define $\widetilde\sigma$ by
\begin{equation}\label{eq:tilde-sigma-definition}
\frac1{\widetilde\sigma}+\frac1\sigma
=
N\left(\frac12-\frac1\rho\right).
\end{equation}
\end{defi}

\begin{prop}[Non-admissible Strichartz estimate in
Lorentz spaces]\label{prop:nonadmissible-strichartz}
Let $I\subset\R$ be an interval, and let $t_0\in\overline I$.  Let $(\sigma,\rho)$ be Lorentz acceptable, and let $\widetilde\sigma$ be defined by \eqref{eq:tilde-sigma-definition}.  Then, for every $1\le \nu_0\le \nu_1\le\infty$,
\begin{equation}\label{eq:Lorentz-nonadmissible-Strichartz}
 \left\|
 \int_{t_0}^{t}U(t-s)f(s)\,ds
 \right\|_{L^\sigma(I;L^{\rho,\nu_1})}
 \le
 C
 \|f\|_{L^{\widetilde\sigma'}(I;L^{\rho',\nu_0})}.
\end{equation}
The constant $C$ depends only on $N,\rho,\sigma,\nu_0$, and $\nu_1$, and is independent of $I$, $t_0$, and $f$.
\end{prop}

We give the proof of Proposition \ref{prop:nonadmissible-strichartz} in Appendix \ref{app:Lorentz-nonadmissible-Strichartz}.

\begin{rem}
Figure \ref{fig:Lorentz-acceptable-range} below illustrates the geometry of Definition \ref{def:Lorentz-acceptable}.  The point $P_{\rm adm}=(1/r_b,1/q_b)$ lies on the admissible line, whereas $P=(1/r_b,1/a_b)$ lies in the Lorentz acceptable region precisely in the range $p_{\mathrm{st}}(b)<p$ used in Theorem \ref{thm:small-s}.  In the proof of Theorem \ref{thm:small-s}, this corresponds to applying Proposition \ref{prop:nonadmissible-strichartz} with $(\sigma,\rho)=(a_b,r_b)$.
\end{rem}

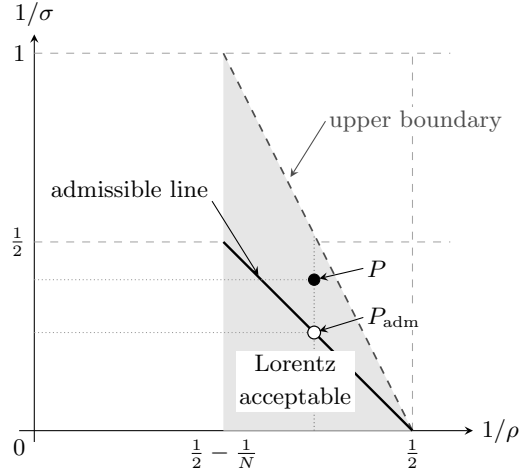
\begin{figure}[!htbp]
\centering
\begin{tikzpicture}[x=10cm,y=5cm,>=stealth,
    point/.style={circle,fill=black,draw=black,inner sep=1.5pt},
    admpt/.style={circle,fill=white,draw=black,line width=0.45pt,inner sep=1.7pt},
    labelbox/.style={fill=white,inner sep=1.2pt}]
    \footnotesize
    \def\xL{0.25}
    \def\xM{0.50}
    \def\xP{0.37}
    \def\yAdm{0.26}
    \def\yP{0.40}
    \def\yTop{0.52}

    \draw[->] (-0.02,0) -- (0.58,0) node[right] {$1/\rho$};
    \draw[->] (0,-0.03) -- (0,1.05) node[above] {$1/\sigma$};

    \draw[dashed,black!35] (0,1) -- (0.55,1);
    \draw[dashed,black!35] (0.5,0) -- (0.5,1);
    \draw[dashed,black!35] (0,0.5) -- (0.55,0.5);

    \fill[black!10] (\xL,0) -- (\xM,0) -- (\xL,1) -- cycle;
    \draw[line width=0.65pt,dashed,black!70] (\xL,1) -- (\xM,0);
    \node[labelbox,text=black!70,anchor=west] (upperlab) at (0.385,0.82) {upper boundary};
    \draw[->,shorten >=2pt,black!70] (upperlab.west) -- (0.335,0.66);

    \draw[line width=0.9pt,black] (\xL,0.5) -- (\xM,0);
    \node[labelbox,text=black,anchor=east] (admlab) at (0.235,0.64) {admissible line};
    \draw[->,shorten >=2pt,black] (admlab.south east) -- (0.300,0.40);

    \draw[densely dotted,black!55] (\xP,0) -- (\xP,\yTop);
    \draw[densely dotted,black!40] (0,\yAdm) -- (\xP,\yAdm);
    \draw[densely dotted,black!40] (0,\yP) -- (\xP,\yP);

    \node[admpt] at (\xP,\yAdm) {};
    \node[point] at (\xP,\yP) {};

    \node[labelbox,anchor=west] (Plab) at (0.435,0.43) {$P$};
    \draw[->,shorten >=2pt] (Plab.west) -- (\xP,\yP);
    \node[labelbox,anchor=west] (Padmlab) at (0.435,0.30) {$P_{\rm adm}$};
    \draw[->,shorten >=2pt] (Padmlab.west) -- (\xP,\yAdm);

    \node[labelbox,align=center] at (0.345,0.13) {Lorentz\\acceptable};

    \node[below left] at (0,0) {$0$};
    \node[below] at (\xL,0) {$\frac12-\frac1N$};
    \node[below] at (0.5,0) {$\frac12$};
    \node[left] at (0,0.5) {$\frac12$};
    \node[left] at (0,1) {$1$};
\end{tikzpicture}
\caption{The Lorentz acceptable region.  Here
$P_{\rm adm}=(1/r_b,1/q_b)$ is the admissible point, while
$P=(1/r_b,1/a_b)$ is the non-admissible point used in the small-data scattering argument. The figure is drawn for $N\ge3$.}
\label{fig:Lorentz-acceptable-range}
\end{figure}
\FloatBarrier

\begin{rem}
If $(q,\rho)$ is admissible and $\sigma>q/2$, then
\[
 \frac1{\widetilde\sigma}+\frac1\sigma=\frac2q.
\]
Thus Proposition \ref{prop:nonadmissible-strichartz} recovers the non-admissible Strichartz estimate established by Aloui and the third author \cite{tayachi2}.  In the proof of Theorem \ref{thm:small-s}, we apply it with $(\sigma,\rho)=(a_b,r_b)$.
\end{rem}

\section{Small-data scattering in the complementary range}\label{sec:scattering}

In this section we prove Theorem \ref{thm:small-s} and Corollary \ref{cor:FcalH1-small-scattering}. Regarding Theorem \ref{thm:small-s}, that is (INLS) with $(A1)$ and $p>p_{\mathrm{st}}(b)$ it is worth noting that the approach developed in \cite{tayachi2} does not rely on any particular structure of the nonlinearity. In the super-Strauss regime, small-data scattering solutions to (INLS) can be constructed by means of the contraction mapping principle, using the non-admissible Strichartz estimates. For completeness, we provide the details below.
The proof is based on a fixed point argument in a scale-invariant spacetime norm, controlled by the non-admissible Strichartz estimate in Lorentz spaces.
The smallness condition is imposed on the free evolution, in the spirit of the approach of Cazenave and Weissler \cite{Cazenavewissler} and Kato's approach \cite{Kato94}.

Fix $b$ satisfying \eqref{bchoice}, and let $q_b$, $r_b$, and $a_b$ be defined by \eqref{main-scatter-exponents}.  Then $(q_b,r_b)$ is an admissible pair and
\[
 \frac1{r_b'}=\frac bN+\frac p{r_b},
 \qquad
 \frac1{q_b'}=\frac{p-1}{a_b}+\frac1{q_b}.
\]
Define $\widetilde a_b$ by
\[
 \frac1{\widetilde a_b}+\frac1{a_b}= \frac{2}{q_b}.
\]
Then
\[
 \widetilde a_b'=\frac{a_b}{p}.
\]
Moreover, under \eqref{bchoice}, the condition $p_{\mathrm{st}}(b)<p$ is equivalent to the Lorentz acceptability of $(a_b,r_b)$ in the sense of Definition \ref{def:Lorentz-acceptable}; equivalently,
\[
 a_b>\frac{q_b}{2}.
\]
Indeed,
\[
 a_b>\frac{q_b}{2}
 \quad\Longleftrightarrow\quad
 N(p-1)^2+(N-2+2b)(p-1)-4+2b>0,
\]
which is equivalent to $p>p_{\mathrm{st}}(b)$.
We shall also use $a_b>p$.

For an interval $I\subset\R$, define
\[
 S(I):=L^{a_b}(I;L^{r_b,\infty}),
 \qquad
 \|u\|_{S(I)}:=\|u\|_{L^{a_b}(I;L^{r_b,\infty})},
\]
and
\[
 W(I):=C(I;L^2)\cap L^{q_b}(I;L^{r_b,2}),
 \qquad
 \|u\|_{W(I)}
 :=
 \|u\|_{L^\infty(I;L^2)}
 +
 \|u\|_{L^{q_b}(I;L^{r_b,2})}.
\]
By the convention fixed in Section \ref{sec:pre}, the spaces $S(I)$ and $W(I)$ are Banach spaces.
For $u$ defined on an interval containing $0$, write
\[
 \mathcal N(u)(t)
 :=
 -i\int_0^t U(t-s)K F(u(s))\,ds.
\]

We first state the nonlinear estimates used in the fixed point argument.

\begin{lem}[Nonlinear estimates in the scattering space]
Assume \eqref{Kup}, \eqref{A1}, and \eqref{bchoice}.  Let $I\subset\R$ be an interval containing $0$.  Then the following estimates hold.
\begin{enumerate}
\item If $u,v\in S(I)$, then
\begin{align}
\|\mathcal N(u)\|_{S(I)}
&\le C\|u\|_{S(I)}^p,\label{S-self-est}\\
\|\mathcal N(u)-\mathcal N(v)\|_{S(I)}
&\le C
\left(\|u\|_{S(I)}^{p-1}+\|v\|_{S(I)}^{p-1}\right)
\|u-v\|_{S(I)}.\label{S-diff-est}
\end{align}

\item If $u,v\in S(I)\cap W(I)$, then
\begin{align}
\|\mathcal N(u)\|_{W(I)}
&\le C\|u\|_{S(I)}^{p-1}\|u\|_{W(I)},\label{W-self-est}\\
\|\mathcal N(u)-\mathcal N(v)\|_{W(I)}
&\le C
\left(\|u\|_{S(I)}^{p-1}+\|v\|_{S(I)}^{p-1}\right)
\|u-v\|_{W(I)}.\label{W-diff-est}
\end{align}
\end{enumerate}
\end{lem}

\begin{proof}
Since $b_0\le b\le b_\infty$, the coefficient bound \eqref{Kup} implies
\[
 |K(x)|\lesssim |x|^{-b}.
\]
Using \eqref{A1} and the generalized H\"older inequality in Lorentz spaces, we obtain
\begin{align*}
 \|K F(u)\|_{L^{r_b',\infty}}
 &\lesssim
 \|u\|_{L^{r_b,\infty}}^p,\\
 \|K(F(u)-F(v))\|_{L^{r_b',\infty}}
 &\lesssim
 \left(
 \|u\|_{L^{r_b,\infty}}^{p-1}
 +\|v\|_{L^{r_b,\infty}}^{p-1}
 \right)
 \|u-v\|_{L^{r_b,\infty}}.
\end{align*}
Together with $\widetilde a_b'=a_b/p$, this gives
\[
 \|K F(u)\|_{L^{\widetilde a_b'}(I;L^{r_b',\infty})}
 \lesssim
 \|u\|_{S(I)}^p
\]
and the corresponding difference estimate.  Proposition \ref{prop:nonadmissible-strichartz}, applied with
\[
(\sigma,\rho)=(a_b,r_b),
\qquad
\nu_0=\nu_1=\infty,
\qquad
\widetilde\sigma=\widetilde a_b,
\]
yields \eqref{S-self-est} and \eqref{S-diff-est}.

We next prove the $W(I)$ estimates.  By \eqref{A1}, the generalized H\"older inequality, and
\[
 \frac1{q_b'}=\frac{p-1}{a_b}+\frac1{q_b},
\]
we have
\begin{align*}
 \|K F(u)\|_{L^{q_b'}(I;L^{r_b',2})}
 &\lesssim
 \|u\|_{S(I)}^{p-1}
 \|u\|_{L^{q_b}(I;L^{r_b,2})},\\
 \|K(F(u)-F(v))\|_{L^{q_b'}(I;L^{r_b',2})}
 &\lesssim
 \left(
 \|u\|_{S(I)}^{p-1}
 +\|v\|_{S(I)}^{p-1}
 \right)
 \|u-v\|_{L^{q_b}(I;L^{r_b,2})}.
\end{align*}
The admissible Strichartz estimate then gives the $L^\infty(I;L^2)$ and $L^{q_b}(I;L^{r_b,2})$ bounds in \eqref{W-self-est} and \eqref{W-diff-est}.  The $L^2$-continuity of the Duhamel term follows from the standard strong-continuity argument, since the nonlinear term belongs to $L^{q_b'}(I;L^{r_b',2})$.
\end{proof}

We now construct the solution by a fixed point argument.

\begin{prop}[Solution under smallness of the free evolution]\label{prop:small-free-fixed-point}
Assume \eqref{Kup}, \eqref{A1}, and \eqref{bchoice}.  There exists $\varepsilon_0>0$ with the following property. Let $I$ be one of $[0,\infty)$, $(-\infty,0]$, and $\R$.  If $u_0\in L^2$ satisfies
\[
 \|U(t)u_0\|_{S(I)}\le \varepsilon_0,
\]
then the Cauchy problem \eqref{11.1} has a unique solution
\[
 u\in C(I;L^2)\cap L^{q_b}(I;L^{r_b,2}) \cap L^{a_b}(I;L^{r_b,\infty})
\]
Moreover,
\[
 \|u\|_{S(I)}
 \le
 2\|U(t)u_0\|_{S(I)}.
\]
\end{prop}

\begin{proof}
We prove the assertion for $I=[0,\infty)$.  The proof on $(-\infty,0]$ is identical.

Let $C_0>0$ be such that
\[
 \|U(t)u_0\|_{W(I)}\le C_0\|u_0\|_{L^2},
\]
which follows from Proposition \ref{prop:admissible-strichartz}.  Set
\[
 M:=2C_0\|u_0\|_{L^2}
\]
and define
\[
 X_{I,\varepsilon,M}
 :=
 \left\{
 u\in S(I)\cap W(I):
 \|u\|_{S(I)}\le 2\varepsilon,
 \quad
 \|u\|_{W(I)}\le M
 \right\}, \quad
 \varepsilon := \|U(t)u_0\|_{S(I)}.
\]
We equip $X_{I,\varepsilon,M}$ with the metric
\[
 d(u,v):=\|u-v\|_{S(I)}+\|u-v\|_{W(I)}.
\]
Then $X_{I,\varepsilon,M}$ is complete.  Define
\[
 \Phi[u](t)
 :=
 U(t)u_0-i\int_0^t U(t-s)K F(u(s))\,ds.
\]
We show that $\Phi$ is a contraction on $X_{I,\varepsilon,M}$ if $\varepsilon_0$ is sufficiently small.

Let $u\in X_{I,\varepsilon,M}$.  By \eqref{S-self-est},
\[
 \|\Phi[u]\|_{S(I)}
 \le
 \|U(t)u_0\|_{S(I)}
 +C\|u\|_{S(I)}^p
 \le
 \varepsilon+C(2\varepsilon)^p.
\]
Choose $\varepsilon_0>0$ so small that
\[
 C2^p\varepsilon_0^{p-1}\le1.
\]
Then $\|\Phi[u]\|_{S(I)}\le2\varepsilon$.  Similarly, by \eqref{W-self-est},
\[
 \|\Phi[u]\|_{W(I)}
 \le
 C_0\|u_0\|_{L^2}
 +C\|u\|_{S(I)}^{p-1}\|u\|_{W(I)}
 \le
 C_0\|u_0\|_{L^2}
 +C(2\varepsilon)^{p-1}M.
\]
Taking $\varepsilon_0$ smaller if necessary, we may assume
\[
 C(2\varepsilon_0)^{p-1}\le\frac12.
\]
Then
\[
 \|\Phi[u]\|_{W(I)}
 \le
 C_0\|u_0\|_{L^2}+\frac12M
 =
 M.
\]
Thus $\Phi$ maps $X_{I,\varepsilon,M}$ into itself.

For $u,v\in X_{I,\varepsilon,M}$, \eqref{S-diff-est} and \eqref{W-diff-est} give
\begin{align*}
 d(\Phi[u],\Phi[v])
 &\le
 C
 \left(
 \|u\|_{S(I)}^{p-1}
 +\|v\|_{S(I)}^{p-1}
 \right)
 d(u,v)\\
 &\le
 2C(2\varepsilon_0)^{p-1}d(u,v).
\end{align*}
Choosing $\varepsilon_0$ even smaller so that
\[
 2C(2\varepsilon_0)^{p-1}\le\frac12,
\]
we obtain
\[
 d(\Phi[u],\Phi[v])\le\frac12d(u,v).
\]
The contraction mapping theorem gives a fixed point $u\in X_{I,\varepsilon,M}$.
This fixed point satisfies the integral equation on $I$ and
\[
 \|u\|_{S(I)}
 \le
 2\varepsilon
 =
 2\|U(t)u_0\|_{S(I)}.
\]

It remains to prove uniqueness in $S(I)\cap W(I)$.
Let $u,v\in S(I)\cap W(I)$ be two solutions to \eqref{11.1} with the same initial data.
Since
\[
 \|u\|_{S(I)}+\|v\|_{S(I)}<\infty,
\]
we may decompose $I$ into finitely many consecutive intervals $I_j$ such that
\[
 C
 \left(
 \|u\|_{S(I_j)}^{p-1}
 +
 \|v\|_{S(I_j)}^{p-1}
 \right)
 \le \frac12
\]
for every $j$.
On the first interval, the difference estimates give
\[
 \|u-v\|_{S(I_1)}+\|u-v\|_{W(I_1)}
 \le
 \frac12
 \left(
 \|u-v\|_{S(I_1)}+\|u-v\|_{W(I_1)}
 \right),
\]
and hence $u=v$ on $I_1$.
Repeating the same argument,
we obtain $u=v$ on each $I_j$.
Thus $u=v$ on $I$.
This completes the proof.
\end{proof}

We now prove the scattering for small free evolution.

\begin{proof}[Proof of Theorem \ref{thm:small-s}]
We prove the assertion for $I=[0,\infty)$.  The case $I=(-\infty,0]$ is treated in the same way.
By Proposition \ref{prop:small-free-fixed-point}, the solution is global on $I$ and satisfies
\[
 u\in S(I)\cap W(I).
\]
By the nonlinear estimate and the admissible Strichartz estimate on compact subintervals of $I$, $u$ is a solution to \eqref{11.1} in the sense of Definition \ref{def:solution}.

For $0<t_1<t_2$, the integral equation gives
\[
 U(-t_2)u(t_2)-U(-t_1)u(t_1)
 =
 -i\int_{t_1}^{t_2}U(-s)K F(u(s))\,ds.
\]
By Proposition \ref{prop:admissible-strichartz} and the estimate used in the proof of \eqref{W-self-est},
\begin{align*}
 \|U(-t_2)u(t_2)-U(-t_1)u(t_1)\|_{L^2}
 &\le
 C\|K F(u)\|_{L^{q_b'}(t_1,t_2;L^{r_b',2})}\\
 &\le
 C\|u\|_{S((t_1,t_2))}^{p-1}
 \|u\|_{L^{q_b}(t_1,t_2;L^{r_b,2})}.
\end{align*}
Since $u\in S(I)\cap L^{q_b}(I;L^{r_b,2})$,
the right-hand side tends to zero as $t_1,t_2\to\infty$.  Hence $U(-t)u(t)$ is Cauchy in $L^2$ as $t\to\infty$.  Therefore, there exists $u_+\in L^2$ such that
\[
 U(-t)u(t)\to u_+
 \quad\text{in }L^2
 \quad\text{as }t\to\infty.
\]
Since $U(t)$ is unitary on $L^2$, we conclude that $u(t) \rightarrow U(t)u_+$ in $L^2$ as $t \rightarrow \infty$.
The same argument on $(-\infty,0]$ gives the backward scattering assertion.

Assume
\[
 \|U(t)u_0\|_{L^{a_b}(\R;L^{r_b,\infty})}\le\varepsilon_0.
\]
Arguing as above and applying Proposition \ref{prop:small-free-fixed-point} with $I=\R$, we obtain a global solution on $\R$ in the sense of Definition \ref{def:solution}.  The preceding Cauchy argument on the two half-lines then yields scattering in $L^2$ as $t\to\pm\infty$.
This completes the proof.
\end{proof}

It remains to show that small data in $\mathcal{F}H^1$ satisfy the smallness condition for the free evolution in Theorem \ref{thm:small-s}.

\begin{lem}\label{lem:linear-FcalH1}
Let $b$ satisfy \eqref{bchoice}.  Then
\[
 \|U(t)u_0\|_{L^{a_b}(\R;L^{r_b,\infty})}
 \le C\|u_0\|_{\mathcal{F}H^1}
\]
for every $u_0\in \mathcal{F}H^1$.
\end{lem}

\begin{proof}
On $|t|\le1$, the admissible Strichartz estimate for $(q_b,r_b)$, the embedding $L^{r_b,2}\hookrightarrow L^{r_b,\infty}$, and $a_b\le q_b$ yield
\[
 \|U(t)u_0\|_{L^{a_b}((-1,1);L^{r_b,\infty})}
 \lesssim
 \|U(t)u_0\|_{L^{q_b}((-1,1);L^{r_b,2})}
 \lesssim
 \|u_0\|_{L^2}
 \le
 \|u_0\|_{\mathcal{F}H^1}.
\]
Here $a_b\le q_b$ follows from $p\le p_0(b)$.

On $|t|\ge1$, the dispersive estimate gives
\[
 \|U(t)u_0\|_{L^{r_b}}
 \lesssim
 |t|^{-\delta}\|u_0\|_{L^{r_b'}},
 \qquad
 \delta
 =
 N\left(\frac12-\frac1{r_b}\right)
 =
 \frac{N(p-1)+2b}{2(p+1)}.
\]
We have $r_b'<2$ since $p>1$ and $b\ge0$.  When $N\ge3$, the assumption $p\le p_0(b)$ implies $r_b'>2N/(N+2)$.
For $N=1,2$, $r_b'>1$ and $2N/(N+2)\le1$.
It then follows from H\"older's inequality that
\[
 \|f\|_{L^{r_{b}'}}
 \le
 \|\langle x\rangle f\|_{L^2}
 \|\langle x\rangle^{-1}\|_{L^{q}}
 \le
 C\|f\|_{\mathcal{F}H^1}, \qquad
  \frac{1}{q} := \frac{1}{r_{b}'} - \frac{1}{2}.
\]
Therefore, since $L^{r_b}\hookrightarrow L^{r_b,\infty}$,
\[
 \|U(t)u_0\|_{L^{a_b}(\{|t|\ge1\};L^{r_b,\infty})}
 \lesssim
 \|u_0\|_{\mathcal{F}H^1}
 \left(\int_1^\infty t^{-a_b\delta}\,dt\right)^{1/a_b}.
\]
The condition $a_b\delta>1$ is equivalent to $p>p_{\mathrm{st}}(b)$.  Hence the last integral is finite.  Combining the estimates on $|t|\le1$ and $|t|\ge1$ proves the desired assertion.
\end{proof}

\begin{proof}[Proof of Corollary \ref{cor:FcalH1-small-scattering}]
By Lemma \ref{lem:linear-FcalH1},
\[
 \|U(t)u_0\|_{L^{a_b}(\R;L^{r_b,\infty})}
 \le
 C\|u_0\|_{\mathcal{F}H^1}.
\]
Choose $\varepsilon_1=\varepsilon_0/C$, where $\varepsilon_0$ is the constant in Theorem \ref{thm:small-s}.  If
\[
 \|u_0\|_{\mathcal{F}H^1}\le\varepsilon_1,
\]
then the smallness condition for the free evolution in Theorem \ref{thm:small-s} holds on $\R$.  Therefore, a global solution exists and scatters in both time directions.
\end{proof}

\section{Proof of Theorem \ref{but}}\label{sec:pfm}

In this section we prove Theorem \ref{but}.  The proof is based on the test-function method developed in \cite{ikedaSobajima,ikedaSobajima2,miyazaki1,miyazakinew}, originating from Zhang's method \cite{zhang1,zhang2}.
More specifically, we follow the unified test-function framework for the case $K\equiv1$ developed by the second author and Sobajima \cite{miyazakinew}.
In the present two-scale setting, the model weight $|x|^{-b}$ is replaced by the coercive lower weight $\kappa(x)$, and the exponent governing the large-scale estimates is $b_\infty$.
We begin with the weak formulation used in the proof.

\begin{defi}[Weak solution]\label{def:weak-solution}
Let $T>0$ and $u_0\in L^2$.  We say that $u$ is a weak solution to
\eqref{11.1} on $[0,T)$ with initial data $u_0$ if
\[
 u\in C([0,T);L^2)\cap L^{q_0}_{\rm loc}((0,T);L^{r_0,2}),
 \qquad
 u(0)=u_0,
\]
\[
 K F(u)\in L^1_{\rm loc}((0,T)\times\R^N),
\]
and
\begin{align*}
&\int_0^T\int_{\R^N}u(t,x)
\left(-i\partial_t\psi(t,x)+\frac12\Delta\psi(t,x)\right)\,dx\,dt \\
={}&
i\int_{\R^N}u_0(x)\psi(0,x)\,dx
+
\int_0^T\int_{\R^N}K(x)F(u(t,x))\psi(t,x)\,dx\,dt
\end{align*}
for every $\psi\in C_0^\infty((-\infty,T)\times\R^N)$.
\end{defi}

Every solution in the sense of Definition \ref{def:solution} satisfies this
weak formulation on every finite time interval.  This follows from the
standard density argument based on Strichartz estimates; see, for example,
\cite{ikeda3}.  In the proof of Theorem \ref{but} below, we use only this
weak formulation, the local integrability of $K F(u)$, and the coercive
lower bound in \eqref{A2}.  The regularity assumption in \eqref{A1} is needed
only to connect the weak formulation with the solution class used elsewhere
in the paper.

We define a cut-off function $\eta$
satisfying
\[
      \eta \in C^{\infty}(\R), \quad
      \mathbbm{1}_{(- \infty, 1/2]} \le \eta(s) \le \mathbbm{1}_{(-\infty,1]}.
\]
For $R>0$, set
\[
\psi_R(t,x)
=
\left\{\eta\!\left(\frac{|x|^2+t}{R^2}\right)\right\}^{2p'},
\qquad
t\ge0,\quad x\in\R^N.
\]
This type of cut-off function was firstly introduced in \cite{MP01} (cf. \cite{ikedaSobajima}).  We shall use the following elementary estimates.

\begin{lem}[Cut-off estimates]\label{lem:cutoff}
Let $R>0$ and $p>1$.  Then
\[
  |\partial_t\psi_R(t,x)|
  \lesssim
  R^{-2}\psi_R(t,x)^{1-\frac1{2p'}},
  \qquad
  |\Delta\psi_R(t,x)|
  \lesssim
  R^{-2}\psi_R(t,x)^{1-\frac1{p'}},
\]
for all $t\ge0$ and $x\in\R^N$.
\end{lem}
To justify this choice of test function, fix $T>R^2$ and choose
$\chi\in C_0^\infty(( -\infty,T))$ such that $\chi\equiv1$ on a
neighborhood of $[0,R^2]$.  Set
\[
 \widetilde{\psi}_R(t,x)
 =
 \chi(t)
 \left\{
   \eta \left(\frac{|x|^2+t}{R^2}\right)
 \right\}^{2p'},
 \qquad
 (t,x)\in(-\infty,T)\times\R^N.
\]
Then
$\widetilde{\psi}_R\in C_0^\infty(( -\infty,T)\times\R^N)$ and
$\widetilde{\psi}_R=\psi_R$ on $0\le t\le R^2$.  Since
$\psi_R=0$ for $t\ge R^2$, applying the weak formulation with
$\widetilde{\psi}_R$
gives
\begin{align}
\begin{aligned}
&\int_0^{R^2}\int_{\R^N} u(t,x)
\left(-i\partial_t \widetilde{\psi}_R(t,x)+\frac12\Delta \widetilde{\psi}_R(t,x)\right)\,dx\,dt \\
&=
i\int_{\R^N}u_0(x) \widetilde{\psi}_R(0,x)\,dx
+
\int_0^{R^2}\int_{\R^N}K(x)F(u(t,x)) \widetilde{\psi}_R(t,x)\,dx\,dt.
\end{aligned}
\label{eq:test-identity}
\end{align}
for $R>1$.
In what follows, we write $\psi_R$ for $\widetilde{\psi}_R$ for simplicity.
We set
\[
I_\kappa(R)
=
\int_0^{R^2}\int_{\R^N}
\kappa(x)|u(t,x)|^p\psi_R(t,x)\,dx\,dt.
\]

We now use the lower bound \eqref{Klower}.

\begin{lem}\label{lem:test-estimate-kappa}
Let $u$ be a weak solution to \eqref{11.1} on $[0,\infty)$.
Assume \eqref{Klower}.  Then
\begin{align*}
\left|
\int_0^{R^2}\int_{\R^N}u(t,x)
\left(-i\partial_t\psi_R(t,x)+\frac12\Delta\psi_R(t,x)\right)\,dx\,dt
\right|
 \lesssim
R^{N-\frac{N+2-b_\infty}{p}}I_\kappa(R)^{1/p}
\end{align*}
for every $R>1$.
\end{lem}

\begin{proof}
By Lemma \ref{lem:cutoff} and $0\le\psi_R\le1$,
\[
\left|
-i\partial_t\psi_R+\frac12\Delta\psi_R
\right|
\lesssim
R^{-2}\psi_R^{1/p}.
\]
Hence, by H\"older's inequality,
\begin{align*}
&\left|
\int_0^{R^2}\int_{\R^N}u(t,x)
\left(-i\partial_t\psi_R(t,x)+\frac12\Delta\psi_R(t,x)\right)\,dx\,dt
\right|\\
&\quad\lesssim
R^{-2}
\int_{\{|x|^2+t\le R^2\}}|u(t,x)|\psi_R(t,x)^{1/p}\,dx\,dt\\
&\quad\le
R^{-2}
\left(
\int_{\{|x|^2+t\le R^2\}}\kappa(x)^{-\frac1{p-1}}\,dx\,dt
\right)^{1/p'}
\left(
\int_{\{|x|^2+t\le R^2\}}\kappa(x)|u(t,x)|^p\psi_R(t,x)\,dx\,dt
\right)^{1/p}\\
&\quad\le
R^{-2}
\left(
\int_{\{|x|^2+t\le R^2\}}\kappa(x)^{-\frac1{p-1}}\,dx\,dt
\right)^{1/p'}
I_\kappa(R)^{1/p}.
\end{align*}
By \eqref{Klower},
\[
\kappa(x)^{-\frac1{p-1}}
\lesssim
|x|^{\frac{b_0}{p-1}}
\langle x\rangle^{\frac{b_\infty-b_0}{p-1}}.
\]
Therefore
\begin{align*}
\int_{\{|x|^2+t\le R^2\}}\kappa(x)^{-\frac1{p-1}}\,dx\,dt
&\lesssim
\int_0^{R^2}\int_{|x|\le R}
|x|^{\frac{b_0}{p-1}}
\langle x\rangle^{\frac{b_\infty-b_0}{p-1}}\,dx\,dt\\
&\lesssim
R^{N+2+\frac{b_\infty}{p-1}}.
\end{align*}
It follows that
\[
R^{-2}
\left(
\int_{\{|x|^2+t\le R^2\}}\kappa(x)^{-\frac1{p-1}}\,dx\,dt
\right)^{1/p'}
\lesssim
R^{N-\frac{N+2-b_\infty}{p}},
\]
which proves the desired estimate.
\end{proof}

We next estimate the term involving the initial data.
In test-function arguments, this term is often handled by imposing a slowly decaying condition on the initial data.
Here we instead use the weighted assumption $u_0\in\mathcal{F}H^\alpha$, which is natural in the scattering setting.
Under the weighted condition, this term turns out to be harmless in our argument.

\begin{lem}\label{lem:initial-data-term}
Let $\alpha>0$.  Then, for every $R>2$,
\[
\left|
\int_{\R^N}u_0(x)\psi_R(0,x)\,dx
\right|
\lesssim
\|u_0\|_{\mathcal{F}H^\alpha}
\begin{cases}
1, & \alpha > N/2, \\
\log R, & \alpha = N/2,\\
R^{\frac{N-2\alpha}{2}}, & \alpha< N/2.
\end{cases}
\]
\end{lem}

\begin{proof}
Since $\operatorname{supp}\psi_R(0,\cdot)\subset B(0,R)$, H\"older's inequality gives
\begin{align*}
\left|
\int_{\R^N}u_0(x)\psi_R(0,x)\,dx
\right|
&\le
\int_{B(0,R)}|u_0(x)|\,dx\\
&\le
\left(
\int_{B(0,R)}
|u_0(x)|^2\langle x\rangle^{2\alpha}\,dx
\right)^{1/2}
\left(
\int_{B(0,R)}
\langle x\rangle^{-2\alpha}\,dx
\right)^{1/2}.
\end{align*}
The last factor is bounded by a constant if $\alpha>N/2$, by $C(\log R)^{1/2}$ if $\alpha=N/2$, and by $CR^{(N-2\alpha)/2}$ if $\alpha<N/2$.  Since $(\log R)^{1/2} \lesssim \log R$ for $R>2$, the desired assertion follows.
\end{proof}

The following lemma extracts a lower bound for the nonlinear spacetime integral from the existence of a nonzero scattering state.

\begin{lem}\label{lem:lower-Ikappa}
Let $u$ be a weak solution to \eqref{11.1} on $[0,\infty)$.
Let $u_+\in L^2$ if $1<p\le2$, and let
$u_+\in L^2\cap L^{\frac{p}{p-1}}$ if $p>2$.  Assume that
$u$ satisfies the corresponding asymptotic condition in Theorem \ref{but},
namely condition (i) if $1<p\le2$ and condition (ii) if $p>2$,
with asymptotic state $u_+$.  If $u_+\not\equiv0$, then
\begin{equation}\label{lowerIkappa}
R^{N+1-b_\infty-\frac N2p}
\lesssim
I_\kappa(R)
\end{equation}
holds for all sufficiently large $R$.
\end{lem}

\begin{proof}
Choose $\rho_{\ast}>1$ such that
\[
\|\widehat{u_+}\|_{L^p(B(0,\rho_{\ast}))}\ne0.
\]
This is possible under the assumptions on $u_+$.  Indeed, if $1<p\le2$, then $\widehat{u_+}\in L^2\subset L^p_{\rm loc}$; if $p>2$, then the Hausdorff--Young inequality gives $\widehat{u_+}\in L^p$.
Since the Fourier transform is injective on $L^2$, $\widehat{u_+}\not\equiv0$. Thus its $L^p$-norm on some ball is nonzero.
For $R>2$, define
\[
D_{\rho_{\ast},R}
=
\left\{
(t,x)\in(0,R^2)\times\R^N:
\frac{R}{4\rho_{\ast}}\le t\le\frac{R}{2\rho_{\ast}},
\quad
|x|\le \rho_{\ast} t
\right\}.
\]
Then $\psi_R\equiv1$ on $D_{\rho_{\ast},R}$ for all sufficiently large $R$.  Moreover, \eqref{Klower} implies
\[
\kappa(x)\gtrsim R^{-b_\infty}
\qquad
\text{on }D_{\rho_{\ast},R}.
\]
We use the standard factorization of the free Schr\"odinger group
\[
U(t)=M(t)D(t)\mathcal{F}M(t),
\]
where
\[
[M(t)f](x)=e^{i|x|^2/(2t)}f(x),
\qquad
[D(t)f](x)=t^{-N/2}f(x/t).
\]
Set
\[
u_{p}(t):=M(t)D(t)\mathcal{F}u_+,
\]
which is the asymptotics of $U(t)u_{+}$.
Then
\[
u
=
u_{p}(t)
+\bigl(U(t)u_+-u_{p}(t)\bigr)
+\bigl(u-U(t)u_+\bigr).
\]
By the triangle inequality in $L^p(D_{\rho_{\ast},R})$,
\begin{align}
\begin{aligned}
I_\kappa(R)^{1/p}
&\ge
\left(
\iint_{D_{\rho_{\ast},R}}\kappa(x)|u(t,x)|^p\,dx\,dt
\right)^{1/p}\\
&\gtrsim
R^{-b_\infty/p}
\Bigg[
\left(
\iint_{D_{\rho_{\ast},R}}|u_{p}(t,x)|^p\,dx\,dt
\right)^{1/p}\\
&\hspace{6em}
-
\left(
\iint_{D_{\rho_{\ast},R}}|U(t)u_+-u_{p}(t)|^p\,dx\,dt
\right)^{1/p}\\
&\hspace{6em}
-
\left(
\iint_{D_{\rho_{\ast},R}}|u-U(t)u_+|^p\,dx\,dt
\right)^{1/p}
\Bigg].
\end{aligned}
\label{eq:triangle-lower}
\end{align}

We estimate the last term in \eqref{eq:triangle-lower}.
Let us first treat the case $1<p\le2$.
Fix $\varepsilon>0$, to be chosen sufficiently small.
By the asymptotic condition in Theorem \ref{but} {\rm (i)}, there exists
$t_*(\varepsilon)>0$ such that
\[
 \|u(t)-U(t)u_+\|_{L^2}<\varepsilon
\]
for all $t\ge t_*(\varepsilon)$.  Hence, by H\"older's inequality,
\begin{align}
\begin{aligned}
\iint_{D_{\rho_{\ast},R}}|u(t)-U(t)u_+|^p\,dx\,dt
&\le
\int_{R/(4\rho_{\ast})}^{R/(2\rho_{\ast})}
\int_{B(0,\rho_{\ast}t)}
|u(t)-U(t)u_+|^p\,dx\,dt \\
&\lesssim
\sup_{t\ge R/(4\rho_{\ast})}
\|u(t)-U(t)u_+\|_{L^2}^p
R^{N+1-\frac N2p} \\
&\le
\varepsilon^p R^{N+1-\frac N2p}
\end{aligned}
\label{eq:error-nonlinear-p-leq-2}
\end{align}
for all $R\ge4\rho_{\ast}t_*(\varepsilon)$.

Next, by the factorization of $U(t)$,
\begin{align*}
\iint_{D_{\rho_{\ast},R}}|U(t)u_+-u_p(t)|^p\,dx\,dt
&=
\iint_{D_{\rho_{\ast},R}}
\left|
D(t)\mathcal{F}\bigl(M(t)u_+-u_+\bigr)
\right|^p\,dx\,dt\\
&=
\int_{R/(4\rho_{\ast})}^{R/(2\rho_{\ast})}
t^{N-\frac N2p}
\int_{B(0,\rho_{\ast})}
\left|
\mathcal{F}\bigl(M(t)u_+-u_+\bigr)
\right|^p\,dx\,dt.
\end{align*}
Since $1<p\le2$, H\"older's inequality on $B(0,\rho_{\ast})$ and Plancherel's theorem imply
\[
\iint_{D_{\rho_{\ast},R}}|U(t)u_+-u_p(t)|^p\,dx\,dt
\lesssim
R^{N+1-\frac N2p}
\left(
\sup_{t\ge R/(4\rho_{\ast})}
\|M(t)u_+-u_+\|_{L^2}
\right)^p.
\]
Since $u_+\in L^2$ and $|M(t)|=1$, Lebesgue's dominated convergence theorem gives
\[
 M(t)u_+\to u_+
 \quad\text{in }L^2
 \quad\text{as }t\to\infty.
\]
Therefore, for all sufficiently large $R$,
\begin{equation}\label{eq:linear-error}
\iint_{D_{\rho_{\ast},R}}|U(t)u_+-u_p(t)|^p\,dx\,dt
\lesssim
\varepsilon^p R^{N+1-\frac N2p}.
\end{equation}

We now consider the case $p>2$.  Set
\[
q=\frac{4p}{N(p-2)},
\qquad
\frac2q+\frac Np=\frac N2.
\]
Fix $\varepsilon>0$, to be chosen sufficiently small.
By the assumption in Theorem \ref{but} (ii), there exists $t_*(\varepsilon)>0$ such that
\[
T^{1/q}
\|u-U(\cdot)u_+\|_{L^q(T,\infty;L^p)}
<\varepsilon
\]
for all $T\ge t_*(\varepsilon)$.  Therefore, for $R\ge4\rho_{\ast}t_*(\varepsilon)$,
\begin{align}
\begin{aligned}
\iint_{D_{\rho_{\ast},R}}|u(t)-U(t)u_+|^p\,dx\,dt
&\le
\int_{R/(4\rho_{\ast})}^{R/(2\rho_{\ast})}
\|u(t)-U(t)u_+\|_{L^p}^p\,dt\\
&\le
\|u-U(\cdot)u_+\|_{L^q(R/(4\rho_{\ast}),\infty;L^p)}^p
R^{1-\frac pq}\\
&\lesssim
\varepsilon^p R^{-\frac pq}R^{1-\frac pq}
=
\varepsilon^p R^{N+1-\frac N2p}.
\end{aligned}
\label{eq:error-nonlinear-p-greater-2}
\end{align}
For the second term in \eqref{eq:triangle-lower}, the Hausdorff--Young inequality gives
\begin{align*}
\iint_{D_{\rho_{\ast},R}}|U(t)u_+-u_{p}(t)|^p\,dx\,dt
&\lesssim
R^{N+1-\frac N2p}
\left(
\sup_{t\ge R/(4\rho_{\ast})}
\|M(t)u_+-u_+\|_{L^{p'}}
\right)^p.
\end{align*}
Since $u_+\in L^{p'}$, we have $M(t)u_+\to u_+$ in $L^{p'}$ as $t\to\infty$.  Hence \eqref{eq:linear-error} holds for all sufficiently large $R$ also in this case.

Finally, the leading term in \eqref{eq:triangle-lower} is explicitly computed as
\begin{align}
\begin{aligned}
\iint_{D_{\rho_{\ast},R}}|u_{p}(t,x)|^p\,dx\,dt
&=
\int_{R/(4\rho_{\ast})}^{R/(2\rho_{\ast})}
t^{N-\frac N2p}
\int_{B(0,\rho_{\ast})}
|\widehat{u_+}(\xi)|^p\,d\xi\,dt \\
&=
C R^{N+1-\frac N2p}
\|\widehat{u_+}\|_{L^p(B(0,\rho_{\ast}))}^p.
\end{aligned}
\label{eq:leading-profile}
\end{align}
Combining \eqref{eq:triangle-lower}, \eqref{eq:error-nonlinear-p-leq-2}, \eqref{eq:error-nonlinear-p-greater-2}, \eqref{eq:linear-error}, and \eqref{eq:leading-profile}, we obtain
\[
I_\kappa(R)^{1/p}
\gtrsim
R^{-b_\infty/p}
R^{\frac1p\left(N+1-\frac N2p\right)}.
\]
This proves \eqref{lowerIkappa}.
\end{proof}

\begin{proof}[Proof of Theorem \ref{but}]
Assume $u_+\not\equiv0$ and argue by contradiction.  Multiplying \eqref{eq:test-identity} by $e^{i\varphi}$ and taking real parts, we see from \eqref{KFpositive} and Lemma \ref{lem:test-estimate-kappa} that
\begin{align*}
c_0I_\kappa(R)
&\le
\operatorname{Re}
\left(
e^{i\varphi}
\int_0^{R^2}\int_{\R^N}
K(x)F(u(t,x))\psi_R(t,x)\,dx\,dt
\right)\\
&=
\operatorname{Re}
\left(
e^{i\varphi}
\int_0^{R^2}\int_{\R^N}
u(t,x)
\left(-i\partial_t\psi_R(t,x)+\frac12\Delta\psi_R(t,x)\right)\,dx\,dt
\right)\\
&\quad
-
\operatorname{Re}
\left(
ie^{i\varphi}
\int_{\R^N}u_0(x)\psi_R(0,x)\,dx
\right)\\
&\lesssim
\left|
\int_{\R^N}u_0(x)\psi_R(0,x)\,dx
\right|
+
R^{N-\frac{N+2-b_\infty}{p}}I_\kappa(R)^{1/p}.
\end{align*}
Young's inequality yields
\[
R^{N-\frac{N+2-b_\infty}{p}}I_\kappa(R)^{1/p}
\le
\varepsilon I_\kappa(R)
+
C_\varepsilon R^{N-\frac{2-b_\infty}{p-1}}
\]
for small $\varepsilon>0$.
Absorbing the first term on the right-hand side, we get
\[
I_\kappa(R)
\lesssim
\left|
\int_{\R^N}u_0(x)\psi_R(0,x)\,dx
\right|
+
R^{N-\frac{2-b_\infty}{p-1}}.
\]
By Lemma \ref{lem:initial-data-term},
\begin{equation}\label{eq:Ikappa-upper}
I_\kappa(R)
\lesssim
\begin{cases}
1+R^{N-\frac{2-b_\infty}{p-1}},
& \alpha> N/2,\\
\log R+R^{N-\frac{2-b_\infty}{p-1}},
& \alpha= N/2,\\
R^{\frac{N-2\alpha}{2}}+R^{N-\frac{2-b_\infty}{p-1}},
& \alpha< N/2.
\end{cases}
\end{equation}
Set
\[
 \gamma:=N-\frac{2-b_\infty}{p-1}.
\]
First consider the case $p>1+(2-b_\infty)/N$.  Then $\gamma>0$.  The assumption $(2-b_\infty)/(p-1) - N/2 \leq \alpha$
gives
\[
 \frac{N-2\alpha}{2}\le \gamma.
\]
Hence \eqref{eq:Ikappa-upper} implies
\[
 I_\kappa(R)\lesssim R^\gamma
\]
for all sufficiently large $R$.
Combining this with Lemma \ref{lem:lower-Ikappa}, we obtain
\[
 R^{N+1-b_\infty-\frac N2p}
 \lesssim
 R^\gamma.
\]
for all sufficiently large $R$.
This leads to
\[
 N+1-b_\infty-\frac N2p
 \le
 N-\frac{2-b_\infty}{p-1},
\]
which is equivalent to $p\ge p_{\mathrm{st}}(b_\infty)$.  This contradicts the assumption
$p<p_{\mathrm{st}}(b_\infty)$.

It remains to consider $1<p\le1+(2-b_\infty)/N$.  Then $\gamma\le0$, and the
condition on $\alpha$ implies $\alpha\ge N/2$.  Thus \eqref{eq:Ikappa-upper} gives
\[
 I_\kappa(R)\lesssim 1+\log R.
\]
On the other hand, Lemma \ref{lem:lower-Ikappa} gives
\[
 R^{N+1-b_\infty-\frac N2p}
 \lesssim
 I_\kappa(R).
\]
Since
\[
 N+1-b_\infty-\frac N2p
 \ge
 \frac{N-b_\infty}{2}
 >0,
\]
this is impossible as $R\to\infty$.  Therefore $u_+\equiv0$.
The assertion for the backward time direction follows by the same argument, after the change of variables $t\mapsto -t$.
\end{proof}

\appendix
\section{Proof of the Lorentz non-admissible Strichartz estimate}
\label{app:Lorentz-nonadmissible-Strichartz}

\begin{proof}[Proof of Proposition \ref{prop:nonadmissible-strichartz}]
We include the proof for completeness.  It combines the dispersive estimate with
O'Neil's convolution inequality in Lorentz spaces.
Set
\[
 \theta=N\left(\frac12-\frac1\rho\right).
\]
Since $(\sigma,\rho)$ is Lorentz acceptable, we have
\[
 0<\theta<1,
 \qquad
 0<\frac1\sigma<\theta.
\]
Hence $\widetilde\sigma$ defined by \eqref{eq:tilde-sigma-definition} satisfies $1<\widetilde\sigma<\infty$.
We first recall the Lorentz dispersive estimate
\begin{equation}\label{eq:Lorentz-dispersive-appendix}
 \|U(t)g\|_{L^{\rho,\nu_1}}
 \le
 C |t|^{-\theta}
 \|g\|_{L^{\rho',\nu_0}},
 \qquad t\ne0,
\end{equation}
for $1\le\nu_0\le\nu_1\le\infty$.  Indeed, the usual dispersive estimate
$L^1\to L^\infty$ and the $L^2$ unitarity of $U(t)$ give \eqref{eq:Lorentz-dispersive-appendix} with $\nu_0=\nu_1$ by real interpolation.
The case $\nu_0\le\nu_1$ follows from the embedding $L^{\rho,\nu_0}\hookrightarrow L^{\rho,\nu_1}$.

We identify $f$ with its zero extension outside $I$ and set
\[
 G(s)=\|f(s)\|_{L^{\rho',\nu_0}}.
\]
For $t\in I$, the Duhamel term is bounded by the full time convolution:
\[
\left\|
\int_{t_0}^{t}U(t-s)f(s)\,ds
\right\|_{L^{\rho,\nu_1}}
\le
C\int_{\R}|t-s|^{-\theta}G(s)\,ds.
\]
The time orientation is irrelevant for this estimate; if $t<t_0$,
we take the absolute value and integrate over the interval between $t$
and $t_0$.
Since $|\cdot|^{-\theta}\in L^{1/\theta,\infty}(\R)$,
O'Neil's convolution
inequality \cite{ONeil} and the embedding
$L^{\sigma,\widetilde\sigma'}(\R)\hookrightarrow L^\sigma(\R)$ for
$\widetilde\sigma'\le\sigma$ yield
\[
 \left\||\cdot|^{-\theta}*G\right\|_{L^\sigma(\R)}
 \lesssim \left\||\cdot|^{-\theta}*G\right\|_{L^{\sigma,\widetilde\sigma'}(\R)}
 \lesssim
 \||\cdot|^{-\theta}\|_{L^{1/\theta,\infty}(\R)}
 \|G\|_{L^{\widetilde\sigma'}(\R)}.
\]
Here the exponents satisfy
\[
 1+\frac1\sigma
 =
 \theta+\frac1{\widetilde\sigma'},
\]
which is equivalent to \eqref{eq:tilde-sigma-definition}.  Restricting this estimate to $I$ gives \eqref{eq:Lorentz-nonadmissible-Strichartz}.
\end{proof}

\section*{Use of AI tools}

The authors used GPT-5.5 (OpenAI) to assist in identifying candidate illustrative examples satisfying the assumptions, flagging possible gaps in selected proofs, and improving the English presentation.
All mathematical content was independently verified by the authors, who take full responsibility for the final manuscript.

\section*{Acknowledgments}
H.M. was supported by JSPS KAKENHI Grant Number 22K13941 and 26K00612.

\end{document}